\documentclass[reqno]{amsart}

\usepackage{amsfonts}
\usepackage{amsmath}
\usepackage{amssymb}

\usepackage{cite}
\usepackage{url}

\allowdisplaybreaks

\newtheorem{theorem}{Theorem}

\newtheorem{definition}[theorem]{Definition}

\newtheorem{lemma}[theorem]{Lemma}
\newtheorem{notation}[theorem]{Notation}
\newtheorem{proposition}[theorem]{Proposition}
\newtheorem{remark}[theorem]{Remark}

\begin{document}

\title[Existence of solution to a thermistor problem]{%
Existence of solution to a nonlocal conformable fractional thermistor problem}

\author[M. R. Sidi Ammi]{Moulay Rchid Sidi Ammi}
\address{M. R. Sidi Ammi: Department of Mathematics,
AMNEA Group, Faculty of Sciences and Techniques,
Moulay Ismail University,
B.P. 509, Errachidia, Morocco.}
\email{sidiammi@ua.pt, rachidsidiammi@yahoo.fr}
\urladdr{http://orcid.org/0000-0002-4488-9070}

\author[D. F. M. Torres]{Delfim F. M. Torres}
\address{D. F. M. Torres:
Center for Research and Development in Mathematics and Applications (CIDMA),
Department of Mathematics, University of Aveiro, 3810-193 Aveiro, Portugal.}
\email{delfim@ua.pt}
\urladdr{http://orcid.org/0000-0001-8641-2505}


\date{Received: June 20, 2018; Revised: July 26 and Aug 08, 2018; Accepted: Aug 13, 2018}

\subjclass[2010]{Primary 35A01; Secondary 26A33.}

\keywords{Existence of solutions, fractional differential equations,
conformable fractional derivatives.}

\thanks{This is a preprint of a paper whose final and definite form is with 
\emph{Commun. Fac. Sci. Univ. Ank. Ser. A1 Math. Stat.}, ISSN 1303-5991, 
available at [\url{http://communications.science.ankara.edu.tr}]. 
Submitted June 20, 2018; Revised: July 26 and Aug 08, 2018; Accepted: Aug 13, 2018.}


\begin{abstract}
We study a nonlocal thermistor problem for fractional derivatives
in the conformable sense. Classical Schauder's fixed point
theorem is used to derive the existence of a tube solution.
\end{abstract}

\maketitle


\section{Introduction}

The fractional calculus may be considered an old and yet novel topic.
It has started from some speculations of Leibniz, in 1695 and 1697,
followed later by Euler, in 1730, and has been strongly developed
till present days \cite{MR2799292,MR2736622}.
In recent years, considerable interest in fractional calculus has been stimulated
by its many applications in several fields of science, including physics,
chemistry, aerodynamics, electrodynamics of complex media, signal
processing, and optimal control \cite{MR3787674,MR2768178,28}.
Most fractional derivatives are defined through fractional integrals
\cite{28,MR3561379,29,29bis}.
Due to this fact, those fractional derivatives inherit a nonlocal behavior,
which leads to many interesting applications, including memory effects
and future dependence \cite{MR2768178,Ref:1:2,66,67,16,17,24}.

In 2014, a new simpler and more efficient definition of fractional
derivative, depending just on the basic limit definition of derivative,
was introduced in \cite{khalil}. See also \cite{MR3293309,chung,unal,MyID:324,eslami}
for further developments on conformable differentiation. The new notion
is prominently compatible and conformable with the classical derivative.
In contrast with other fractional definitions, this new concept satisfies
more standard formulas for the derivative of the product and quotient
of two functions and has a simpler chain rule. In addition to the conformable
derivative, the conformable fractional integral has been also introduced,
and Rolle and mean value theorems for conformable fractional differentiable
functions obtained. The subject is nowadays under strong development
\cite{MR3641366,MyID:379,MR3805277,MR3815617}. This is well explained
by the fact that the new definition reflects a natural extension
of the usual derivative to solve different types of fractional differential equations
\cite{ahmed,emrah,rochdi}. The main advantages of the conformable calculus,
among others, are: (i) the simple nature of the conformable fractional derivative,
which allows for many extensions of classical theorems in calculus
(e.g., product and chain rules) that are indispensable in applications
but not valid for classical fractional differential models;
(ii) the conformable fractional derivative of a constant function is zero,
which is is not the case for some fractional derivatives like the Riemann--Liouville;
(iii) conformable fractional derivatives, conformable chain rule,
conformable integration by parts, conformable Gronwall's inequality,
conformable exponential function, conformable Laplace transform,
all tend to the corresponding ones in usual calculus;
(iv) while in the standard calculus there exist functions that do not
have Taylor power series representations about certain points,
in the theory of conformable calculus they do have;
(v) a nondifferentiable function can be
differentiable in the conformable sense.

The thermistor concept was first discovered in 1833 by Michael Faraday (1791--1867),
who noticed that the silver sulfides resistance decreases as the temperature increases.
This lead Samuel Ruben (1900--1988) to invent the first commercial thermistor
in the 1930s. Roughly speaking, a thermistor is a circuit device that may
be used either as a current limiting device or a
temperature sensing device. Typically, it is a small cylinder
made from a ceramic material whose electrical conductivity
depends strongly on the temperature.
The heat produced by an electrical current, passing through a
conductor device, is governed by the so-called thermistor equations.
Nowadays, thermistors can be found everywhere, in airplanes, air conditioners,
cars, computers, medical equipment, hair dryers, portable heaters,
incubators, electrical outlets, refrigerators, digital thermostats,
ice sensors and aircraft wings, ovens, stove tops and in all kinds of appliances.
Knowing it, it is not a surprise that a great deal of attention
is currently paid, by many authors, to the study of thermistor problems
\cite{MR3806233,MR3810479,MR3767245}. In \cite{sidiammi1},
existence and uniqueness of a positive solution to generalized
nonlocal thermistor problems with fractional-order derivatives were discussed.
Recently, Sidi Ammi et al. studied global existence
of solutions for a fractional Caputo nonlocal thermistor problem \cite{MR3736617}.
Existence and uniqueness results for a fractional
Riemann--Liouville nonlocal thermistor problem on arbitrary time scales
are investigated in \cite{MyID:365}, while dynamics and stability results
for Hilfer fractional-type thermistor problems are studied in \cite{hilfer}.
The Hilfer fractional derivative has been used to interpolate both
the Riemann--Liouville and the Caputo fractional derivative.

While previous works assume the electrical conductivity to be a smooth
and bounded function from above and below, or a Lipschitz continuous function
depending strongly in both time and temperature, in contrast, here
we only use the hypothesis of continuity on the electrical conductivity.
Motivated by the results of \cite{hammoudi}, we establish
existence of a tube solution for a conformable fractional nonlocal
thermistor problem by means of Schauder's fixed point theorem.
More precisely, we are concerned with heat conduction in a thermistor
used as a current surge regulator governed by the following equations:
\begin{equation}
\label{p1}
\begin{split}
u^{(\alpha)}(t) & = \frac{\lambda f(t, u(t))}{\left(
\int_{a}^{T}f(x, u(x))\, dx\right)^{2}}= g(t, u(t)) \, ,
\quad  t \in  [a, T]  \, , \\
u(a)& =u_a,
\end{split}
\end{equation}
where $u$ describes the temperature of the conductor
and $u^{(\alpha)}(t)$ denotes the conformable fractional
derivative of $u$ at $t$ of order $\alpha$, $\alpha\in(0,1)$.
We assume that $a$, $T$ and $\lambda$ are fixed positive reals.
Moreover, as already mentioned, we assume the following hypothesis:
\begin{itemize}
\item[$(H_1)$] $f: [a, T] \times \mathbb{R}^{+}\rightarrow \mathbb{R}^{+*}$
is a continuous function.
\end{itemize}

The rest of the article is arranged as follows.
In Section~\ref{sec:2}, we give preliminary definitions
and set the basic concepts and necessary results from the
simple and interesting conformable fractional calculus.
Then, in Section~\ref{sec:3}, we prove
existence of a tube solution via Schauder's
fixed point theorem.


\section{Preliminaries}
\label{sec:2}

We first recall the definition of conformable
fractional derivative as given in \cite{khalil}.

\begin{definition}[Conformable fractional derivative \cite{khalil}]
\label{def:cfd}
Let $\alpha\in(0,1)$ and $f : [0,\infty)\rightarrow \mathbb{R}$. The
conformable fractional derivative of $f$ of order $\alpha$ is defined by
$$
T_{\alpha}(f)(t):=\lim_{\epsilon\rightarrow0}
\frac{f(t+\epsilon t^{1-\alpha})-f(t)}{\epsilon}
$$
for all $t>0$. Often, we write $f^{(\alpha)}$ instead of $T_{\alpha}(f)$
to denote the conformable fractional derivative of $f$ of order $\alpha$.
In addition, if the conformable fractional derivative of $f$ of order
$\alpha$ exists, then we simply say that $f$ is $\alpha$-differentiable.
If $f$ is $\alpha$-differentiable in some $t \in (0,a)$, $a>0$, and
$\lim_{t\rightarrow 0^{+}}f^{(\alpha)}(t)$ exists, then we
define $f^{(\alpha)}(0):=\lim_{t\rightarrow 0^{+}}f^{(\alpha)}(t)$.
\end{definition}

If $f$ is differentiable at a point $t>0$, then
$T_{\alpha}(f)(t)=t^{1-\alpha}\frac{df}{dt}(t)$.

\begin{remark}
If $f \in C^1$, then one has
$$
\lim_{\alpha\rightarrow 1} T_\alpha(f)(t) = f'(t),
\quad
\lim_{\alpha\rightarrow 0} T_\alpha(f)(t) = t f'(t).
$$
\end{remark}

\begin{definition}[Conformable fractional integral \cite{khalil}]
Let $\alpha\in(0,1)$, $f : [a,\infty)\rightarrow \mathbb{R}$. The
conformable fractional integral of $f$ of order $\alpha$
from $a$ to $t$, denoted by $I_{\alpha}^{a}(f)(t)$, is defined by
$$
I_{\alpha}^{a}(f)(t)
:=\int_{a}^{t}\frac{f(\tau)}{\tau^{1-\alpha}}d\tau,
$$
where the above integral is the usual improper Riemann integral.
\end{definition}

\begin{theorem}[See \cite{khalil}]
\label{th1}
If $f$ is a continuous function in the domain of $I_{\alpha}^{a}$, then
$T_{\alpha}\left(I_{\alpha}^{a}(f)\right)(t)=f(t)$ for all $t\geq a$.
\end{theorem}

\begin{notation}
Let $0 < a < b$. We denote by ${_{\alpha}\mathfrak{J}}_{a}^{b}[f]$
the value of the integral $\int_{a}^{b}\frac{f(t)}{t^{1-\alpha}}dt$,
that is, ${_{\alpha}\mathfrak{J}}_{a}^{b}[f] := I_{\alpha}^{a}(f)(b)$.
We also denote by $C^{(\alpha)}([a,b],\mathbb{R})$, $0<a<b$, $\alpha>0$,
the set of all real-valued functions $f:[a,b]\rightarrow \mathbb{R}$ that are
$\alpha$-differentiable and for which the $\alpha$-derivative is continuous.
We often abbreviate $C^{(\alpha)}([a,b],\mathbb{R})$ by $C^{(\alpha)}([a,b])$.
\end{notation}

\begin{lemma}[See \cite{hammoudi}]
\label{le}
Let $r\in C^{(\alpha)}([a,b])$, $0<a<b$, be such that $r^{(\alpha)}(t)<0$
on the set $\left\{ t\in[a,b] : r(t)>0\right\}$. If $r(a)\leq 0$,
then $r(t)\leq 0$ for every $t\in[a,b]$.
\end{lemma}

\begin{theorem}[See \cite{hammoudi}]
\label{th2}
If $g\in L^{1}([a,b])$, then function $x:[a,b]\rightarrow \mathbb{R}$ defined by
\begin{equation*}
x(t):=e^{-\frac{1}{\alpha}\left(\frac{t}{a}\right)^{\alpha}}\left(
e^{\frac{1}{\alpha}}x_{0} + {_{\alpha}\mathfrak{J}}_{a}^{t}\left[
\frac{g(s)}{e^{-\frac{1}{\alpha}(\frac{s}{a})^{\alpha}}}\right]\right)
\end{equation*}
is solution to problem
\begin{equation*}
\begin{cases}
x^{(\alpha)}(t)+\frac{1}{a^{\alpha}}x(t)=g(t), & t\in[a,b], \quad a>0,\\
x(a)=x_{0}.
\end{cases}
\end{equation*}
\end{theorem}

\begin{proposition}[See \cite{hammoudi}]
\label{pr3}
If $x:(0,\infty)\rightarrow \mathbb{R}$
is $\alpha$-differentiable at $t\in[a,b]$, then
$$
|x(t)|^{(\alpha)}=\frac{x(t) \, x^{\alpha}(t)}{|x(t)|}.
$$
\end{proposition}

For proving our main results, we make use
of the following auxiliary definition and lemmas.

\begin{definition}[See p.~112 of \cite{MR1987179}]
Let $X$, $Y$ be topological spaces.
A map $f:X\rightarrow Y$ is called compact if $f(X)$
is contained in a compact subset of $Y$.
\end{definition}

\begin{lemma}[See \cite{li}]
\label{lem2.2}
Let $M$ be a subset of $C([0,T])$. Then $M$ is precompact
if and only if the following conditions  hold:
\begin{enumerate}
\item $\{u(t):u \in M\}$ is uniformly bounded,

\item $\{u(t):u \in M\}$ is equicontinuous on $[0,T]$.
\end{enumerate}
\end{lemma}

\begin{lemma}[Schauder fixed point theorem \cite{li}]
\label{lem2.3 }
Let $U$ be a closed bounded convex subset of a Banach space $X$. If
$T:U\to U$ is completely continuous, then $T$ has a fixed point in $U$.
\end{lemma}


\section{Main Results}
\label{sec:3}

We begin by introducing the notion of tube solution for problem \eqref{p1}.

\begin{definition}
\label{definition}
Let $(v,M)\in C^{(\alpha)}([a,T],\mathbb{R})\times C^{(\alpha)}([a,T],[0,\infty))$.
We say that $(v,M) $ is a tube solution to problem \eqref{p1} if
\begin{enumerate}
\item[(i)] $\left(y-v(t)\right) \left(g(t,y)-v^{(\alpha)}\right)
\leq M(t)M^{(\alpha)}(t)$ for every $t\in[a,b]$
and every real number $y$ such that $|y-v(t)|=M(t)$;

\item[(ii)] $v^{(\alpha)}(t)=g(t,v(t))$ and $M^{(\alpha)}(t)=0$
for all $t\in[a,b]$ such that $M(t)=0$;

\item[(iii)] $|u_{a}-v(a)|\leq M(a)$.
\end{enumerate}
\end{definition}
\begin{notation}
We introduce the following notation:
$$
\mathbf{T}(v,M) := \left\{u\in C^{(\alpha)}([a,T],\mathbb{R}) :
| u(t)-v(t)| \leq M(t), \  t\in [a,T]\right\}.
$$
\end{notation}

Consider the following problem:
\begin{equation}
\label{eq:probAux}
\begin{cases}
u^{(\alpha)}+\frac{1}{a^{\alpha}}u(t)
=g(t,\widetilde{u}(t))+\frac{1}{a^{\alpha}}\widetilde{u}(t),
& t\in[a,T], \quad a>0,\\
u(a)=u_{a},
\end{cases}
\end{equation}
where
\begin{equation}
\label{eq:probAux:wt}
\widetilde{u}(t) :=
\begin{cases}
\frac{M(t)}{|u-v(t)|}(u(t)-v(t))+v(t)& \text{ if } |u-v(t)|> M(t),\\
u(t) & \text{ otherwise}.
\end{cases}
\end{equation}
In order to apply Schauder's fixed point theorem,
let us define the operator $\mathbf{K}:C([a,T])\rightarrow C([a,T])$ by
$$
\mathbf{K}(u)(t) := e^{-\frac{1}{\alpha}(\frac{t}{a})^{\alpha}}\left(
e^{\frac{1}{\alpha}}u_{a}+ {_{\alpha}\mathfrak{J}}_{a}^{t}\left[
\frac{g(s,\widetilde{u}(s))+\frac{1}{a^{\alpha}}\widetilde{u}(s)}{
e^{-\frac{1}{\alpha}(\frac{s}{a})^{\alpha}}}\right]\right).
$$

\begin{proposition}
\label{pr2}
If $(v,M)\in C^{(\alpha)}([a,T],\mathbb{R})\times C^{(\alpha)}([a,T],[0,\infty))$
is a tube solution to \eqref{p1}, then
$\mathbf{K}:C([a,T])\rightarrow C([a,T])$ is compact
and problem \eqref{eq:probAux}--\eqref{eq:probAux:wt}
has a solution.
\end{proposition}

\begin{proof}
Let $\epsilon > 0$ and $\{u_{n} \}_{n\in\mathbb{N}}$
be a sequence of $C([a,T],\mathbb{R})$
that converges to $u\in C([a,T],\mathbb{R})$.
Remark that $L(t)= e^{-\frac{1}{\alpha}(\frac{t}{a})^{\alpha}}$
is a decreasing function on $[a,T]$. Then, $L(T) \leq L(t) \leq L(a)$
for all $t\in [a, T]$. It results that
\begin{equation*}
\begin{split}
|\mathbf{K}(u_{n}(t)) &- \mathbf{K}(u(t))|
=\Bigg|e^{-\frac{1}{\alpha}(\frac{t}{a})^{\alpha}}\left( e^{\frac{1}{\alpha}}u_{a}
+ {_{\alpha}\mathfrak{J}}_{a}^{t}\left[\frac{g(s,\widetilde{u}_{n}(s))
+ \frac{1}{a^{\alpha}}\widetilde{u}_{n}(s)}{e^{-\frac{1}{\alpha}(
\frac{s}{a})^{\alpha}}}\right]\right)\\
& \qquad -e^{-\frac{1}{\alpha}\left(\frac{t}{a}\right)^{\alpha}}
\left( e^{\frac{1}{\alpha}}u_{a}
+ {_{\alpha}\mathfrak{J}}_{a}^{t}\left[\frac{g(s,\widetilde{u}(s))
+ \frac{1}{a^{\alpha}}\widetilde{u}(s)}{
e^{-\frac{1}{\alpha}\left(\frac{s}{a}\right)^{\alpha}}}\right]\right) \Bigg|\\
&\leq \frac{L(a)}{L(T)} {_{\alpha}\mathfrak{J}}_{a}^{t}\left[\left|
\left(g(s,\widetilde{u}_{n}(s))+ \frac{1}{a^{\alpha}}\widetilde{u}_{n}(s)\right)
- \left(g(s,\widetilde{u}(s)) + \frac{1}{a^{\alpha}}\widetilde{u}(s)\right)
\right|\right]\\
&\leq \frac{L(a)}{L(T)} \left({_{\alpha}\mathfrak{J}}_{a}^{t}
\left[\left|g(s,\widetilde{u}_{n}(s))-g(s,\widetilde{u}(s))\right|\right]
+\frac{1}{a^{\alpha}}{_{\alpha}\mathfrak{J}}_{a}^{t}
\left[\left|\widetilde{u}_{n}(s)-\widetilde{u}(s)\right|\right]\right)
\end{split}
\end{equation*}
or
\begin{equation*}
\begin{split}
g(s,\widetilde{u}_{n}(s))- g(s,\widetilde{u}(s))
&= \frac{\lambda f(s, \widetilde{u}_{n}(s))}{\left(
\int_{a}^{T}f(x, \widetilde{u}_{n}(x))\, dx\right)^{2}}
- \frac{\lambda f(s, \widetilde{u}(s))}{\left(
\int_{a}^{T}f(x, \widetilde{u}(x))\, dx\right)^{2}}\\
&= \frac{\lambda}{\left(\int_{a}^{T}f(x, \widetilde{u}_{n})\, d x\right)^{2}}
\left(f(s, \widetilde{u}_{n}(s))-f(s, \widetilde{u}(s))\right) \\
&+ \lambda f(s, \widetilde{u}(s)) \left( \frac{1}{\left( \int_{a}^{T}
f(x, \widetilde{u}_{n})\, d x\right)^{2}}
- \frac{1}{\left(\int_{a}^{T}f(x, \widetilde{u})\,  d x \right)^{2}} \right)\\
&= I_{1} + I_{2}.
\end{split}
\end{equation*}
Since there is a constant $R>0$ such that
$\|\widetilde{u}\|_{C([a,T],\mathbb{R})}<R$, there exists an index
$N$ such that $\|\widetilde{u}_{n}\|_{C([a,T],\mathbb{R})}\leq R$
for all $n>N$. Thus, $f$ is uniformly continuous and, consequently,
uniformly bounded on $[a,T]\times B_{R}(0)$. Then, there exist
positive constants $A$ and $B$ such that $A \leq f(s, v) \leq B$
for all $(s, v) \in [a,T]\times B_{R}(0)$.
Thus, for a well chosen $D$, which will be given below, one has
$$
\exists \eta > 0, \quad
|\widetilde{u}_{n} -\widetilde{u}| < \eta, \quad \forall x \in [a, T],
\quad \left | f(x, \widetilde{u}_{n})\,  - f(x, \widetilde{u})\right | < D
$$
and
\begin{equation*}
\begin{split}
|I_{1}| &\leq \frac{\lambda}{A^{2}(T-a)^{2}} \left|
f(s, \widetilde{u}_{n}(s))-f(s, \widetilde{u}(s)) \right| \\
& \leq \frac{\lambda D}{A^{2}(T-a)^{2}}.
\end{split}
\end{equation*}
Furthermore, we have
\begin{equation*}
\begin{split}
|I_{2}|
&\leq
\frac{\lambda B \left | \left( \int_{a}^{T}f(x, \widetilde{u}_{n})\, d x\right)^{2}
- \left( \int_{a}^{T}f(x, \widetilde{u})\, d x\right)^{2}\right|}{\left(
\int_{a}^{T}f(x, \widetilde{u}_{n})\, d x\right)^{2}
\left( \int_{a}^{T}f(x, \widetilde{u})\, d x\right)^{2}}\\
&\leq \frac{\lambda  B}{A^{4}(T-a)^{4}} \left | \left(
\int_{a}^{T}f(x, \widetilde{u}_{n})\, - f(x, \widetilde{u}) d \tau \right)
\left( \int_{a}^{T}f(x, \widetilde{u}_{n})\,  +f(x, \widetilde{u}) d \tau \right) \right |\\
& \leq \frac{2 \lambda B^{2}}{A^{4}(T-a)^{3}} \left( \int_{a}^{T}
\left | f(x, \widetilde{u}_{n})\,  - f(x, \widetilde{u}) \right | d \tau \right) \\
& \leq \frac{2\lambda B^{2} D}{A^{4}(T-a)^{2}}.
\end{split}
\end{equation*}
Then,
\begin{equation*}
\begin{split}
|I_{1} + I_{2}| & \leq \lambda D \left ( \frac{1}{A^{2}(T-a)^{2}}
+ \frac{2B^{2}}{A^{4}(T-a)^{2}} \right ):= E
\end{split}
\end{equation*}
and it follows that
\begin{equation*}
\begin{split}
|\mathbf{K}(u_{n}(t)) - \mathbf{K}(u(t))|
& \leq \frac{L(a)}{L(T)} \left ( {_{\alpha}\mathfrak{J}}_{a}^{T}(E)
+ \frac{1}{a^{\alpha}} {_{\alpha}\mathfrak{J}}_{a}^{T}(\eta) \right).
\end{split}
\end{equation*}
On the other hand, we can estimate the right
hand side of the above inequality by
\begin{equation*}
\begin{split}
\frac{L(a)}{L(T)} {_{\alpha}\mathfrak{J}}_{a}^{T}(E)
& \leq \frac{L(a)}{L(T)}E {_{\alpha}\mathfrak{J}}_{a}^{T}(1)\\
& \leq \frac{L(a)}{L(T)}E \frac{T^{\alpha} - a^{\alpha}}{\alpha}
= \frac{\epsilon}{2}
\end{split}
\end{equation*}
and
\begin{equation*}
\begin{split}
\frac{L(a)}{L(T)} \frac{1}{a^{\alpha}}{_{\alpha}\mathfrak{J}}_{a}^{T}(\eta)
& \leq \frac{L(a)}{L(T)} \frac{\eta}{a^{\alpha}}{_{\alpha}\mathfrak{J}}_{a}^{T}(1)\\
& \leq \frac{L(a)}{L(T)} \frac{\eta}{a^{\alpha}} \frac{T^{\alpha} - a^{\alpha}}{\alpha}
= \frac{\epsilon}{2}.
\end{split}
\end{equation*}
If we set
$$
E= \frac{\epsilon}{2} \frac{L(T)}{L(a)} \frac{\alpha}{T^{\alpha} - a^{\alpha}}
=\frac{\alpha \epsilon  L(T)}{2 L(a) T^{\alpha} - a^{\alpha}}
$$
and choose
$$
D= \frac{E}{\lambda} \left ( \frac{1}{A^{2}(T-a)^{2}}
+ \frac{2B^{2}}{A^{4}(T-a)^{2}} \right )^{-1}
$$
and
$$
\eta=\frac{\alpha \epsilon a^{\alpha}
L(T)}{2L(a)(T^{\alpha} - a^{\alpha})},
$$
then
$$
|\mathbf{K}(u_{n}(t)) - \mathbf{K}(u(t))|
\leq \epsilon.
$$
This proves the continuity of $\mathbf{K}$.
To finish the proof of Proposition~\ref{pr2},
we prove three technical lemmas.

\begin{lemma}
If $f$ is locally Lipschitzian,
then the operator $\mathbf{K}$ is continuous.
\end{lemma}

\begin{proof}
It is a direct consequence of the inequality
\begin{equation*}
\begin{split}
|\mathbf{K}(u_{n}(t)) - \mathbf{K}(u(t))|
& \leq c \| u_{n}(t) - u(t)\|,
\end{split}
\end{equation*}
which tends to zero as $n$ goes to $+\infty$.
\end{proof}

\begin{lemma}
\label{lem:16}
The set $\mathbf{K}( C([a,T]))$ is s uniformly bounded.
\end{lemma}

\begin{proof}
Let $u_{n}\in C([a,b])$. We have
\begin{equation*}
\begin{split}
\left|\mathbf{K}(u_{n})(t)\right|
&=\left|e^{-\frac{1}{\alpha}(\frac{t}{a})^{\alpha}}\left(
e^{\frac{1}{\alpha}}u_{a} + {_{\alpha}\mathfrak{J}}_{a}^{t}\left[
\frac{g(s,\widetilde{u}_{n}(s)) + \frac{1}{a^{\alpha}}\widetilde{u}_{n}(s)}{
e^{-\frac{1}{\alpha}\left(\frac{s}{a}\right)^{\alpha}}}\right]\right)\right|\\
&\leq L(a)\left(e^{\frac{1}{\alpha}}|u_{a}|+\frac{1}{K(T)}
{_{\alpha}\mathfrak{J}}_{a}^{T} \left[\left|g(s,\widetilde{u}_{n}(s))
+ \frac{1}{a^{\alpha}}\widetilde{u}_{n}(s)\right|\right]\right)\\
&\leq L(a)\left(e^{\frac{1}{\alpha}}|u_{a}| + \frac{1}{L(T)}
{_{\alpha}\mathfrak{J}}_{a}^{T}\left[
\left|g\left(s,\widetilde{u}_{n}(s)\right)\right|\right]
+\frac{1}{L(T) a^{\alpha}} {_{\alpha}\mathfrak{J}}_{a}^{T}
\left[\left|\widetilde{u}_{n}(s)\right|\right]\right).
\end{split}
\end{equation*}
Similarly to above, there is an $R>0$ such that $|\widetilde{u}_{n}(s)|\leq R$
for all $s\in[a,T]$ and all $n\in\mathbb{N}$. Since function $f$ is compact
on $[a,T]\times B_{R}(0)$, it is uniformly bounded and, as a consequence,
$g$ is also uniformly bounded. We deduce that
\begin{equation*}
\begin{split}
\left | g(s, \widetilde{u}_{n} ) \right |
& \leq \frac{\lambda \left |f(s, \widetilde{u}_{n}(s))\right |}{
\left( \int_{a}^{T} f(s, \widetilde{u}_{n}(s) ) ds \right )^{2}}\\
& \leq \frac{\lambda B}{A^{2}(T-a)^{2}} = G.
\end{split}
\end{equation*}
This ends the proof of Lemma~\ref{lem:16}.
\end{proof}

\begin{lemma}
The set $\mathbf{K}( \mathbb(C([a, T]))$ is equicontinuous.
\end{lemma}

\begin{proof}
For $t_{1},t_{2} \in [a,T]$, we have
\begin{equation*}
\begin{split}
|\mathbf{K}(u_{n})(t_{2}) &- \mathbf{K}(u_{n})(t_{1})|\\
& = \left|e^{-\frac{1}{\alpha}(\frac{t_{2}}{a})^{\alpha}}\left(
e^{\frac{1}{\alpha}}u_{a} + {_{\alpha}\mathfrak{J}}_{a}^{t_{2}}\left[
\frac{g(s,\widetilde{u}_{n}(s)) + \frac{1}{a^{\alpha}}\widetilde{u}_{n}(s)}{
e^{-\frac{1}{\alpha}\left(\frac{s}{a}\right)^{\alpha}}}\right]\right)\right. \\
& - \left. e^{-\frac{1}{\alpha}(\frac{t_{1}}{a})^{\alpha}}\left(
e^{\frac{1}{\alpha}}u_{a} + {_{\alpha}\mathfrak{J}}_{a}^{t_{1}}\left[
\frac{g(s,\widetilde{u}_{n}(s)) + \frac{1}{a^{\alpha}}\widetilde{u}_{n}(s)}{
e^{-\frac{1}{\alpha}\left(\frac{s}{a}\right)^{\alpha}}}\right]\right)\right|\\
&\leq e^{\frac{1}{\alpha}}|u_{a}| \left|e^{-\frac{1}{\alpha}(\frac{t_{1}}{a})^{\alpha}}
-e^{-\frac{1}{\alpha}(\frac{t_{2}}{a})^{\alpha}}\right|
+  \left | {_{\alpha}\mathfrak{J}}_{t_{1}}^{t_{2}}\left[
\frac{g(s,\widetilde{u}_{n}(s)) + \frac{1}{a^{\alpha}}\widetilde{u}_{n}(s)}{
e^{-\frac{1}{\alpha}\left(\frac{s}{a}\right)^{\alpha}}}\right] \right | \\
&\leq e^{\frac{1}{\alpha}}|u_{a}| \left|e^{-\frac{1}{\alpha}(\frac{t_{1}}{a})^{\alpha}}
-e^{-\frac{1}{\alpha}(\frac{t_{2}}{a})^{\alpha}}\right|
+ \frac{1}{L(T)} \left | {_{\alpha}\mathfrak{J}}_{t_{1}}^{t_{2}} ( G+\frac{R}{a^{\alpha}}) \right |\\
&\leq e^{\frac{1}{\alpha}}|u_{a}| \left|e^{-\frac{1}{\alpha}(\frac{t_{1}}{a})^{\alpha}}
-e^{-\frac{1}{\alpha}(\frac{t_{2}}{a})^{\alpha}}\right|
+ \frac{1}{L(T)} ( G+\frac{R}{a^{\alpha}})
\left | {_{\alpha}\mathfrak{J}}_{t_{1}}^{t_{2}} (1) \right|\\
&\leq e^{\frac{1}{\alpha}}|u_{a}| \left|e^{-\frac{1}{\alpha}(\frac{t_{1}}{a})^{\alpha}}
-e^{-\frac{1}{\alpha}(\frac{t_{2}}{a})^{\alpha}}\right|  + \frac{1}{L(T)} (
G+\frac{R}{a^{\alpha}}) \left|t_{1}^{\alpha}-t_{2}^{\alpha}\right|.
\end{split}
\end{equation*}
The right hand of the above inequality does not depend on $u$
and tends to zero as $t_{2} \rightarrow t_{1}$.
This proves that the sequence $(\mathbf{K}(u_{n}))_{n\in \mathbb{N}}$ is equicontinuous.
\end{proof}

By the Arzel\`{a}--Ascoli theorem, which asserts that a subset is relatively compact
if and only if it is bounded and equicontinuous
\cite[p.~607]{MR1987179}, $\mathbf{K}(C([a,b]))$ is relatively compact
and therefore  $\mathbf{K}$ is compact. Consequently,
by the Schauder fixed point theorem,
it has a fixed point (see \cite{MR1987179}),
which is a solution to problem \eqref{eq:probAux}--\eqref{eq:probAux:wt}.
We have just proved Proposition~\ref{pr2}.
\end{proof}

We are now ready to state the main result of the paper.

\begin{theorem}
\label{thm:mr}
If $(v,M)\in C^{(\alpha)}([a,T],\mathbb{R})\times C^{(\alpha)}([a,T],[0,\infty))$
is a tube solution to \eqref{p1}, then problem \eqref{p1} has a solution
$u\in C^{(\alpha)}([a,T],\mathbb{R})\cap \mathrm{T}(v,M)$.
\end{theorem}

\begin{proof}
It remains to show that for every solution $u$
to problem \eqref{eq:probAux}--\eqref{eq:probAux:wt},
$u\in\mathbf{T}(v,M)$.
We argue as in \cite{hammoudi}. Consider the set
$A :=\left\{t\in[a,T]:|u(t)-v(t)|>M(t)\right\}$.
If $t\in A$, then by Proposition~\ref{pr3} one has
$$
\left(|u(t)-v(t)|-M(t)\right)^{(\alpha)}
=\frac{\left( u(t)-v(t)\right) \left(u^{(\alpha)}(t)-v^{(\alpha)}(t)
\right)}{|u(t)-v(t)|}-M^{(\alpha)}(t).
$$
Thus, since $(v,M)$ is a tube solution to problem \eqref{p1},
we have on $\{ t\in A : M(t)> 0\}$ that
\begin{equation*}
\begin{split}
(|u(t)&-v(t)|-M(t))^{(\alpha)}\\
&= \frac{\left(u(t)-v(t)\right) \left(u^{(\alpha)}(t)-v^{(\alpha)}(t)
\right)}{|u(t)-v(t)|} - M^{(\alpha)}(t)\\
&= \frac{\left(u(t)-v(t)\right) \left(g(t,\widetilde{u}(t))
+\left(\frac{1}{a^{\alpha}}\widetilde{u}(t)-\frac{1}{a^{\alpha}}u(t)\right)
-v^{(\alpha)}(t)\right)}{|u(t)-v(t)|}-M^{(\alpha)}(t)\\
&= \frac{\left(\widetilde{u}(t)-v(t)\right)
\left(g(t,\widetilde{u}(t))- v^{(\alpha)}(t)\right)}{M(t)}
+\frac{\left(\widetilde{u}(t)-v(t)\right) \left(\widetilde{u}(t)-u(t)
\right)}{a^{\alpha}M(t)}-M^{(\alpha)}(t)\\
&=\frac{\left(\widetilde{u}(t)-v(t)\right)
\left(g(t,\widetilde{u}(t))- v^{(\alpha)}(t)\right)}{M(t)}\\
&\quad +\left[\frac{M(t)}{|u(t)-v(t)|}-1\right]\frac{|u(t)-v(t)|^{2}}{a^{\alpha}
\left|u(t)-v(t)\right|}-M^{(\alpha)}(t)\\
&=\frac{\left(\widetilde{u}(t)-v(t)\right) \left(g(t,\widetilde{x}(t))
- v^{(\alpha)}(t)\right)}{M(t)}+\left[\frac{M(t)}{a^{\alpha}}
-\frac{|u(t)-v(t)|}{a^{\alpha}}\right]- M^{(\alpha)}(t)\\
&\leq \frac{M(t)M^{\alpha}(t)}{M(t)}+\frac{1}{a^{\alpha}}
\left[ M(t)-|u(t)-v(t)|\right]-M^{(\alpha)}(t)\\
&< 0.
\end{split}
\end{equation*}
On the other hand, by Definition~\ref{definition},
we have on $t\in\{ \tau \in A : M(\tau)= 0\}$ that
\begin{equation*}
\begin{split}
(|u(t)- & v(t)|-M(t))^{(\alpha)}\\
&=\frac{\left(u(t)-v(t)\right) \left(g(t,\widetilde{u}(t))
+\left(\frac{1}{a^{\alpha}}\widetilde{u}(t)
-\frac{1}{a^{\alpha}}u(t)\right)
-v^{(\alpha)}(t)\right)}{|u(t)-v(t)|} - M^{(\alpha)}(t)\\
&= \frac{\left(u(t)-v(t)\right) \left(g(t,\widetilde{u}(t))
-v^{(\alpha)}(t) \right)}{|u(t)-v(t)|}
-\frac{1}{a^{\alpha}}|u(t)-v(t)|-M^{(\alpha)}(t)\\
&< -M^{(\alpha)}(t)\\
&=0.
\end{split}
\end{equation*}
If we set
$r(t):= |u(t)-v(t)|-M(t)$, then $r^{(\alpha)}<0$ on $A:= \{t\in [a,T]:r(t)>0\}$.
Moreover, $r(a)\leq 0$ since $u$ satisfies $|u_{a}-v(a)|\leq M(a)$.
It follows from Lemma~\ref{le} that $A=\emptyset$. Therefore, $u\in \mathrm{T}(v,M)$
and the proof of the theorem is complete.
\end{proof}


\section*{Acknowledgements}

This research was carried out while
Sidi Ammi was visiting the Department
of Mathematics of University of Aveiro, Portugal, on May 2018.
The hospitality of the host institution is here gratefully acknowledged.
The authors were supported by the \emph{Center for Research
and Development in Mathematics and Applications} (CIDMA)
of the University of Aveiro, through Funda\c{c}\~ao
para a Ci\^encia e a Tecnologia (FCT),
within project UID/MAT/04106/2013.
They are grateful to two anonymous referees
for valuable comments and suggestions.




\begin{thebibliography}{90}

\bibitem{MR2799292}
Machado, J. T., Kiryakova, V. and Mainardi, F.,
A poster about the old history of fractional calculus,
\textit{Fract. Calc. Appl. Anal.} 13 (2010), no.~4, 447--454.

\bibitem{MR2736622}
Machado, J. T., Kiryakova, V. and Mainardi, F.,
Recent history of fractional calculus,
\textit{Commun. Nonlinear Sci. Numer. Simul.} 16 (2011), no.~3, 1140--1153.

\bibitem{MR3787674}
Agarwal, R. P., Baleanu, Du., Nieto, J. J., Torres, D. F. M. and Zhou, Y.,
A survey on fuzzy fractional differential and optimal control nonlocal evolution equations,
\textit{J. Comput. Appl. Math.} 339 (2018), 3--29.
{\tt arXiv:1709.07766}

\bibitem{MR2768178}
Ortigueira, M. D.,
{\it Fractional calculus for scientists and engineers},
Lecture Notes in Electrical Engineering, 84, Springer, Dordrecht, 2011.

\bibitem{28}
Podlubny, I.,
{\it Fractional differential equations},
Mathematics in Science and Engineering, 198,
Academic Press, Inc., San Diego, CA, 1999.

\bibitem{MR3561379}
Baleanu, D., Diethelm, K., Scalas, E. and Trujillo, J. J.,
{\it Fractional calculus},
Series on Complexity, Nonlinearity and Chaos, 5,
World Scientific Publishing Co. Pte. Ltd., Hackensack, NJ, 2017.

\bibitem{29}
Miller, K. S. and Ross, B.,
{\it An introduction to the fractional calculus and fractional differential equations},
A Wiley-Interscience Publication, John Wiley \& Sons, Inc., New York, 1993.

\bibitem{29bis}
Samko, S. G., Kilbas, A. A. and Marichev, O. I.,
{\it Fractional integrals and derivatives},
translated from the 1987 Russian original,
Gordon and Breach Science Publishers, Yverdon, 1993.

\bibitem{Ref:1:2}
Baleanu, D., Golmankhaneh, A. K., Golmankhaneh, A. K. and Nigmatullin, R. R.,
Newtonian law with memory,
\textit{Nonlinear Dynam.} 60 (2010), no.~1--2, 81--86.

\bibitem{66}
Caputo, M.,
Linear models of dissipation whose $Q$ is almost frequency independent. II,
\textit{Geophys. J. R. Astr. Soc.} 13 (1967), no.~5, 529--539.

\bibitem{67}
Caputo, M. and Mainardi, F.,
Linear models of dissipation in anelastic solids,
\textit{Riv. Nuovo Cimento} (Ser. II) 1 (1971) 161--198.

\bibitem{16}
Hilfer, R.,
{\it Applications of fractional calculus in physics},
World Scientific Publishing Co., Inc., River Edge, NJ, 2000.

\bibitem{17}
Kilbas, A. A., Srivastava, H. M. and Trujillo, J. J.,
{\it Theory and applications of fractional differential equations},
North-Holland Mathematics Studies, 204, Elsevier Science B.V., Amsterdam, 2006.

\bibitem{24}
Mainardi, F.,
Fractional calculus: some basic problems in continuum and statistical mechanics,
in {\it Fractals and fractional calculus in continuum mechanics (Udine, 1996)}, 291--348,
CISM Courses and Lect., 378, Springer, Vienna, 1997.

\bibitem{khalil}
Khalil, R., Al Horani, M., Yousef. A. and Sababheh, M.,
A new definition of fractional derivative,
\textit{J. Comput. Appl. Math.} 264 (2014), 65--70.

\bibitem{MR3293309}
Abdeljawad, T.,
On conformable fractional calculus,
\textit{J. Comput. Appl. Math.} 279 (2015), 57--66.

\bibitem{chung}
Chung, W. S.
Fractional Newton mechanics with conformable fractional derivative,
\textit{J. Comput. Appl. Math.} 290 (2015), 150--158.

\bibitem{unal}
\"{U}nal, E., G\"{o}kdogan, A. and \c{C}elik, E.,
Solutions of sequential conformable fractional differential
equations around an ordinary point and conformable fractional
Hermite differential equation,
\textit{British J. Appl. Science \& Tech.} 10 (2015), 1--11.

\bibitem{MyID:324}
Benkhettou, N., Hassani, S. and Torres, D. F. M.,
A conformable fractional calculus on arbitrary time scales,
\textit{J. King Saud Univ. Sci.} 28 (2016), no.~1, 93--98.
{\tt arXiv:1505.03134}

\bibitem{eslami}
Eslami, M. and Rezazadeh, H.,
The first integral method for Wu-Zhang system with conformable time-fractional derivative,
\textit{Calcolo} 53 (2016), no.~3, 475--485.

\bibitem{MR3641366}
Lazo, M. J. and Torres, D. F. M.,
Variational calculus with conformable fractional derivatives,
\textit{IEEE/CAA J. Autom. Sin.} 4 (2017), no.~2, 340--352.
{\tt arXiv:1606.07504}

\bibitem{MyID:379}
Bayour, B., Hammoudi, A. and Torres, D. F. M.,
A truly conformable calculus on time scales,
\textit{Glob. Stoch. Anal.} 5 (2018), no.~1, 1--14.
{\tt arXiv:1705.08928}

\bibitem{MR3805277}
Feng, Q. and Meng, F.,
Oscillation results for a fractional order dynamic equation
on time scales with conformable fractional derivative,
\textit{Adv. Difference Equ.} 2018, 2018:193.

\bibitem{MR3815617}
Gholami, Y. and Ghanbari, K.,
New class of conformable derivatives and applications
to differential impulsive systems,
\textit{SeMA J.} 75 (2018), no.~2, 305--333.

\bibitem{ahmed}
Kareem, A. M.,
Conformable fractional derivatives
and it is applications for solving fractional differential equations,
\textit{IOSR J. Math.} 13 (2017), 81--87.

\bibitem{emrah}
\"{U}nal, E. and G\"{o}kdogan, A.,
Solution of conformable fractional ordinary differential equations
via differential transform method,
\textit{Optik -- Int. J. Light and Elect. Optics}, 128 (2017), 264--273.

\bibitem{rochdi}
Rochdi, K.,
Solution of some conformable fractional differential equations,
\textit{Int. J. Pure Appl. Math.} 103 (2015), no.~4, 667--673.

\bibitem{MR3806233}
Bartosz, K., Janiczko, T., Szafraniec, P. and Shillor, M.,
Dynamic thermoviscoelastic thermistor problem with contact
and nonmonotone friction,
\textit{Appl. Anal.} 97 (2018), no.~8, 1432--1453.

\bibitem{MR3810479}
Hrynkiv, V. and Turchaninova, A.,
Analytical solution of a one-dimensional thermistor problem with Robin boundary condition,
\textit{Involve} 12 (2019), no.~1, 79--88.

\bibitem{MR3767245}
Mbehou, M.,
The theta-Galerkin finite element method for coupled systems resulting
from microsensor thermistor problems,
\textit{Math. Methods Appl. Sci.} 41 (2018), no.~4, 1480--1491.

\bibitem{sidiammi1}
Sidi Ammi, M. R. and Torres, D. F. M.,
Existence and uniqueness of a positive solution to generalized
nonlocal thermistor problems with fractional-order derivatives,
\textit{Differ. Equ. Appl.} 4 (2012), no.~2, 267--276.
{\tt arXiv:1110.4922}

\bibitem{MR3736617}
Sidi Ammi, M. R., Jamiai, I. and Torres, D. F. M.,
Global existence of solutions for a fractional Caputo nonlocal thermistor problem,
\textit{Adv. Difference Equ.} 2017 (2017), no.~363, 14~pp.
{\tt arXiv:1711.00143}

\bibitem{MyID:365}
Sidi Ammi, M. R. and Torres, D. F. M.,
Existence and uniqueness results for a fractional
Riemann-Liouville nonlocal thermistor problem on arbitrary time scales,
\textit{J. King Saud Univ. Sci.} 30 (2018), no.~3, 381--385.
{\tt arXiv:1703.05439}

\bibitem{hilfer}
Vivek, D., Kanagarajan, K., Sivasundaram, S.,
Dynamics and stability results for Hilfer fractional type thermistor problem,
\textit{Fractal Fract.} 1 (2017), 5, 14~pp.

\bibitem{hammoudi}
Bayour, B. and Torres, D. F. M.,
Existence of solution to a local fractional nonlinear differential equation,
\textit{J. Comput. Appl. Math.} 312 (2017), 127--133.
{\tt arXiv:1601.02126}

\bibitem{MR1987179}
Granas, A. and Dugundji, J,
{\it Fixed point theory},
Springer Monographs in Mathematics,
Springer, New York, 2003.

\bibitem{li}
Li, C. and Sarwar, S.,
Existence and continuation of solutions for Caputo
type fractional differential equations,
\textit{Electron. J. Differential Equations} 2016 (2016), Paper No.~207, 14~pp.

\end{thebibliography}
\end{document}